\documentclass[a4paper]{amsart}

\usepackage{amscd,amsthm,amsfonts,amssymb,amsmath}
\usepackage[colorlinks=true]{hyperref}
\usepackage[all]{xy}
\usepackage{mathrsfs}
\usepackage{longtable}

\newcommand{\BC}{{\mathbb {C}}} 
 \newcommand{\BF}{{\mathbb {F}}}
\newcommand{\BG}{{\mathbb {G}}}

\newcommand{\BQ}{{\mathbb {Q}}}

 \newcommand{\BZ}{{\mathbb {Z}}}

\newcommand{\CE}{{\mathcal {E}}} 
\newcommand{\CG}{{\mathcal {G}}} 
\newcommand{\CI}{{\mathcal {I}}} 
 \newcommand{\CL}{{\mathcal {L}}}
\newcommand{\CM}{{\mathcal {M}}} 
\newcommand{\CO}{{\mathcal {O}}}

\newcommand{\fa}{{\mathfrak{a}}} \newcommand{\fb}{{\mathfrak{b}}}
\newcommand{\fc}{{\mathfrak{c}}} 
 \newcommand{\ff}{{\mathfrak{f}}}

\newcommand{\fm}{{\mathfrak{m}}} 
 \newcommand{\fp}{{\mathfrak{p}}}
\newcommand{\fq}{{\mathfrak{q}}}

 \newcommand{\fB}{{\mathfrak{B}}}
\newcommand{\fC}{{\mathfrak{C}}} \newcommand{\fD}{{\mathfrak{D}}}
 
\newcommand{\fG}{{\mathfrak{G}}} 
\newcommand{\fI}{{\mathfrak{I}}} 
 
\newcommand{\fM}{{\mathfrak{M}}} 
 \newcommand{\fP}{{\mathfrak{P}}}
 \newcommand{\fR}{{\mathfrak{R}}}

\newcommand{\msg}{\mathscr{G}}
\newcommand{\msl}{\mathscr{L}}

\newcommand{\End}{{\mathrm{End}}}

\newcommand{\Gal}{{\mathrm{Gal}}}

\newcommand{\idele}{id\'{e}le~}

\newcommand{\Ker}{{\mathrm{Ker}}}

\renewcommand{\mod}{\, \mathrm{mod}\,\, }

\newcommand{\Neron}{N\'{e}ron~}

\newcommand{\ord}{{\mathrm{ord}}}
\newcommand{\ov}{\overline}

\newcommand{\pair}[1]{\langle {#1} \rangle}

\renewcommand{\Re}{{\mathrm{Re}}}

\newcommand{\sfrac}[2]{\left(\frac {#1}{#2}\right)}

\newcommand{\wt}{\widetilde}                        
\newcommand{\wh}{\widehat}                      
\newcommand{\wpair}[1]{\left\{{#1}\right\}}

\def\mat(#1,#2,#3,#4){
  \begin{pmatrix}
  #1 & #2 \\ #3 & #4
  \end{pmatrix}
}

\newcommand{\thmref}[1]{Theorem~\ref{#1}}

\newcommand{\lemref}[1]{Lemma~\ref{#1}}
\newcommand{\propref}[1]{Proposition~\ref{#1}}
\newcommand{\corref}[1]{Corollary~\ref{#1}}

\theoremstyle{plain}
\newtheorem{thm}{Theorem}[section] \newtheorem{cor}[thm]{Corollary}
\newtheorem{lem}[thm]{Lemma}  \newtheorem{prop}[thm]{Proposition}

\theoremstyle{definition}

\numberwithin{equation}{section}

\begin{document}

\begin{abstract}
Let $K = \BQ(\sqrt{-q})$, where $q$ is any  prime number congruent to $7$ modulo $8$, and let $\CO$ be the ring of integers of $K$. The prime $2$ splits in $K$, say $2\CO = \fp \fp^\ast$, and there is a unique $\BZ_2$-extension $K_\infty$ of $K$ which is unramified outside $\fp$. Let $H$ be the Hilbert class field of $K$, and write $H_\infty = HK_\infty$.  Let $M(H_\infty)$ be the maximal abelian $2$-extension of $H_\infty$ which is unramified outside the primes above $\fp$, and put $X(H_\infty) = \Gal(M(H_\infty)/H_\infty)$. We prove that $X(H_\infty)$ is always a finitely generated $\BZ_2$-module, by an elliptic analogue of Sinnott's cyclotomic argument. We then use this result to prove for the first time the weak $\fp$-adic Leopoldt conjecture for the compositum $J_\infty$ of $K_\infty$ with arbitrary quadratic extensions $J$ of $H$. We also prove some new cases of the finite generation of the Mordell-Weil group $E(J_\infty)$ modulo torsion of certain  elliptic curves $E$ with complex multiplication by $\CO$.
\end{abstract}


\title[Iwasawa's $\mu=0$ conjecture and the weak Leopoldt conjecture]{Analogues of Iwasawa's $\mu=0$ conjecture and the weak Leopoldt conjecture for a non-cyclotomic $\BZ_2$-extension}
\author{Junhwa Choi}
\address{Department of Mathematics, POSTECH, Pohang, Republic of Korea}
\email{jhchoi.math@gmail.com}

\author{Yukako Kezuka}
\address{Fakult\"at f\"ur Mathematik, Universit\"at Regensburg, Regensburg, Germany}
\email{Yukako.Kezuka@mathematik.uni-regensburg.de}

\author{Yongxiong Li}
\address{Yau Mathematical Sciences Center, Tsinghua University, Beijing, China}
\email{yongxiongli@math.tsinghua.edu.cn}


\maketitle
\tableofcontents

\section{Introduction}

Let $K = \BQ(\sqrt{-q})$, where $q$ is any prime number congruent to $7$ modulo $8$, let $\CO$ be the ring of integers of $K$, and let $H$ be the Hilbert class field of $K$. We fix once and for all an embedding of $K$ into $\BC$. By the theory of complex multiplication, we have $H = K(j(\CO))$ where $j$ is the classical modular function; in particular, this fixes an embedding of $H$ into $\BC$. We write $G$ for the Galois group of $H$ over $K$, and $h$ for the class number of $K$. Then $h$ is odd because $K$ has prime discriminant. The prime $p = 2$ splits in $K$, and we write
$$
2\CO = \fp \fp^\ast.
$$
Throughout the remainder of the paper, we fix once and for all an embedding $\iota_\fp$ of $\ov{K}$ into $\BC_2$, which induces the prime $\fp$. By global class field theory, $K$ has a unique $\BZ_2$-extension which is unramified outside $\fp$, which we denote by $K_\infty/K$. Note that the prime $\fp$ is totally ramified in $K_\infty$ because $h$ is odd. Define
$$
H_\infty = HK_\infty, ~\Gamma = \Gal(H_\infty/H).
$$
We write $M(H_\infty)$ for the maximal abelian $2$-extension of $H_\infty$ which is unramified outside the primes of $H_\infty$ above $\fp$, and put
$$
X(H_\infty) = \Gal(M(H_\infty)/H_\infty).
$$
Note that $M(H_\infty)$ is clearly Galois over $H$, and thus $\Gamma$ acts continuously on $X(H_\infty)$ in the usual fashion via inner automorphisms. This action endows $X(H_\infty)$ with the structure of a module over the Iwasawa algebra $\Lambda(\Gamma)$ of $\Gamma$. In fact, it has long been known that $X(H_\infty)$ is a finitely generated torsion module over $\Lambda(\Gamma)$. In the present paper, we prove the following stronger theorem, which is equivalent to saying that the Iwasawa $\mu$-invariant for the $\Lambda(\Gamma)$-module $X(H_\infty)$ vanishes.

\begin{thm}\label{thm1.1}
The Galois group $X(H_\infty)$ is a finitely generated $\BZ_2$-module.
\end{thm}

\noindent  In \S 5, we give some interesting numerical computations of the group $X(H_\infty)$, which show, somewhat surprisingly, that in fact it is zero for all primes $q < 500$ with $q \equiv 7 \mod 8$ except $q = 431$. Moreover, by a simple application of Nakayama's lemma, we obtain the following corollary of Theorem \ref{thm1.1}. Let $J$ denote any quadratic extension of the Hilbert class field $H$,  and write $J_\infty = J K_\infty$. Let $M(J_\infty)$ be the maximal abelian $2$-extension of $J_\infty$ which is unramified outside the primes above $\fp$, and put $X(J_\infty) = \Gal(M(J_\infty)/J_\infty)$.

\begin{cor}\label{cor1.2}
For every quadratic extension $J$ of $H$, the Galois group $X(J_\infty)$ is a finitely generated $\BZ_2$-module.
\end{cor}

\noindent We point out that this corollary implies in particular (see \cite{JC1}) that the weak $\fp$-adic Leopoldt conjecture is valid for the $\BZ_2$-extension $J_\infty/J$. This is the first example where such a weak $\fp$-adic Leopoldt conjecture has been proven for extensions of $K$ which are not in general abelian over $K$.

In the proof of \thmref{thm1.1}, we shall make use of what is, in some sense, the simplest elliptic curve with complex multiplication by $\CO$, which was introduced by Gross \cite{GR1}. He has proven that there exists a unique elliptic curve defined over $\BQ(j(\CO))$, which we shall denote by $A$, whose $j$-invariant is equal to $j(\CO)$, whose ring of $H$-endomorphisms is equal to $\CO$, whose minimal discriminant ideal in $H$ is equal to $(-q^3)$, and which is isogenous to all of its conjugates under the Galois action of $G$. The Gr\"ossencharacter of $A$ is the Hecke character $\psi$ of $H$ with conductor $\fq = (\sqrt{-q})$, which, on ideals $\fa$ of $H$ prime to $\fq$ is defined by the formula
$$
\psi(\fa) = \alpha,~\text{where $(\alpha) = N_{H/K}\fa$ and $\alpha$ is a square modulo $\fq$.}
$$
We have the identity
\begin{equation}\label{1.1}
\psi = \phi \circ N_{H/K}
\end{equation}
where $\phi$ is the following Gr\"ossencharacter of $K$ with conductor $\fq$. Let $B$ denote the abelian variety over $K$, which is the restriction of scalars  from $H$ to $K$ of $A$.  Let $T = \End_K(B)\otimes \BQ$. Then $T$ is an extension of degree $h$ of $K$, and for ideals $\fb$ of $K$ prime to $\fq$ we have $\phi(\fb) = \beta$, where $\beta$ is the unique element of $T$ such that $\fb^h= (\beta^h)$ and $\beta^h$ is a square modulo $\fq$. In particular, we remark that $B$ is isomorphic over $H$ to the product of the elliptic curves $A^\sigma$, where $\sigma$ runs over the element of $G$. It follows that for any integral ideal $\fb$ of $K$ prime to $\fq$, the endomorphism $\phi(\fb)$ of $B/K$ defines a unique isogeny defined over $H$
\begin{equation}\label{1.2}
\eta_{A^\sigma}(\fb): A^\sigma \rightarrow A^{\sigma\sigma_\fb}
\end{equation}
where $\sigma_\fb$ denotes the Artin symbol of $\fb$ in $G$, whose kernel is $A_\fb^\sigma$. For a more detailed discussion on this isogeny, see \cite{GR1}. Moreover, Gross \cite{GR2} has proven that $A$ has a global minimal Weierstrass equation over $H$. Hence we fix once and for all such a global minimal equation
\begin{equation}\label{1.3}
y^2 + a_1xy + a_3y = x^3 + a_2x^2 + a_4x + a_6,
\end{equation}
with coefficients $a_i$ which are all integers in $H$. Moreover, there is an interesting new application of Corollary \ref{cor1.2}  to certain quadratic twists of $A$. Let $E$ denote a twist of $A$ by an arbitrary quadratic extension of $H$, whose conductor is relatively prime to $2q$. Define
\begin{equation}\label{1.4}
F = H(E_{\fp^2}),~F_\infty = H(E_{\fp^\infty}),~\CG = \Gal(F_\infty/H),~\Delta = \Gal(F/H).
\end{equation}
The fields $F$ and $F_\infty$ are, of course, abelian extensions of $H$, but we stress that they are not in general abelian over $K$. Here $\Delta$ is cyclic of order 2, and it can easily be seen that
$F_\infty = FK_\infty$. We recall that the $\fp^\infty$-Selmer group of $E$ over $F_\infty$ is defined by
$$
S_{\fp^\infty}(E/F_\infty) = \Ker\left(H^1(F_\infty, E_{\fp^\infty}) \to \prod_{v}H^1(F_{\infty,v}, E)(\fp)\right)
$$
where $v$ runs over all finite places of $F_\infty$.

\begin{thm}\label{thm1.3}
Let $E$ be a twist of $A$ by any quadratic extension of $H$ of conductor prime to $2q$. Then the Pontrjagin dual of  $S_{\fp^\infty}(E/F_\infty)$ is always a finitely generated $\BZ_2$-module. In particular, both $E(F_\infty)$ and $E(H_\infty)$ modulo torsion are finitely generated abelian groups.
\end{thm}

\noindent We stress that this result was unknown previously, except in the very special case when $E$ is the quadratic twist of $A$ by the compositum with H of a quadratic extension of $K$. Moreover, none of the analytic results is known for such a curve $E$, for example, the construction of the $\fp$-adic $L$-function attached to $E$.

Our proof of Theorem \ref{thm1.1} uses an elliptic analogue of Sinnott's beautiful proof of the vanishing of the cyclotomic $\mu$-invariant. Considerable past work in this direction has already been done  by Gillard \cite{Gi2}, \cite{Gi1} and Schneps \cite{Sch} for split odd primes $p$. For the prime $p=2$ there has recently been independent work by Oukhaba and Vigui\'e \cite{OV}, which would seemingly include a proof of Theorem \ref{thm1.1}.  However, we give the full details of a rather different construction of the $\fp$-adic $L$-functions and the analogue of Sinnott's proof in our case, rather than the arguments sketched in \cite{OV}.

\medskip
\section{Construction of the $\fp$-adic $L$-function}

The aim of this section is to construct the $\fp$-adic $L$-function attached to the curve $A/H$, by using the method of \cite{CW3} and \cite{CG}. In this section,  $F$ and $F_\infty$ will denote the fields defined by \eqref{1.4} in the special case in which $E = A$. Thus, we will have
$$
F = H(A_{\fp^2}),~F_\infty = H(A_{\fp^\infty}).
$$
Write $\chi_\fp: \CG \rightarrow \CO_\fp^\times = \BZ_2^\times$ for the character giving the action of $\CG = \Gal(F_\infty/H)$ on $A_{\fp^\infty}$. It is well-known that $A$ has good reduction everywhere over $F$ (the proof of Lemma 2.1 of \cite{CC} generalizes immediately to all of the curves $A$ discussed here). Thus we must have $[F:H] =2$, whence we see that $\chi_\fp$ is an isomorphism. Hence $\CG = \Gamma \times \Delta$, where $\Delta = \Gal(F/H)$ is of order $2$ and $\Gamma = \Gal(F_\infty/F)$ is isomorphic to $\BZ_2$. Note that all the primes of $H$ lying above the  $\fq$ must be ramified in the extension $F/H$, because $A/H$ has bad reduction at the primes of $H$ above $\fq$. As $H_\infty/H$ is unramified outside of the primes of $H$ dividing $\fp$, we see that $H_\infty \cap F = H$, whence we can also identify $\Gamma$ with the Galois group of $H_\infty/H$ under restriction.

Recall that $G$ denotes the Galois group of $H$ over $K$. Let $\fa$ be any non-zero integral ideal of K, which we will always assume is prime to $\fp\fq$. We write $\sigma_\fa$ for the  Artin symbol of $\fa$ in $G$, and $A^\fa$ for the image of $A$ under $\sigma_\fa$.  A global minimal Weierstrass equation for $A^\fa/H$,  and its associated \Neron differential $\omega^\fa$, are given respectively by just applying $\sigma_\fa$ to the coefficients of the equation \eqref{1.3}, and to the coefficients of its \Neron differential $\omega = dx/(2y+a_1x+a_3)$. We then define an element $\xi(\fa)$ of $H$ by the equation
\begin{equation}\label{2.1}
\eta_A(\fa)^\ast(\omega^\fa) = \xi(\fa)\omega.
\end{equation}
where $\eta_A(\fa): A \to A^\fa$ is the isogeny as defined in \eqref{1.2}.
We also write $\CL$ and $\CL_\fa$ for the period lattices of  $\omega$ and $\omega_\fa$. If we write $\CL = \Omega_\infty \CO$ where $\Omega_\infty \in \BC$, then we have $\CL_\fa = \xi(\fa)\Omega_\infty\fa^{-1}$. Note that the Weierstrass isomorphism $\fM(z,\CL_\fa)$ from $\BC/\CL_\fa$ to $A^\fa(\BC)$ is given by
$$
\left(\wp(z,\CL_\fa)- \frac{a_{1,\fa}^2+4a_{2,\fa}}{12},~\frac{1}{2}\left(\wp'(z,\CL_\fa)-a_{1,\fa}\left(\wp(z,\CL_\fa)-\frac{a_{1,\fa}^2+4a_{2,\fa}}{12}\right)-a_{3,\fa} \right)\right)
$$
where we simply write $a_{i,\fa}$ for $\sigma_\fa(a_i)$, and where $\wp(z,\CL_\fa)$ denotes the Weierstrass $\wp$-function of the lattice $\CL_\fa$.

Let $P = (x,y)$ denote a generic point of our global minimal Weierstrass equation for $A^\fa/H$. Given any non-zero element $\lambda$ of $\CO = \End_H(A)$ with $\lambda \neq \pm 1$ and $(\lambda, 6\fq) = 1$,  we define the rational function $R_{\lambda, \fa}(P)$ on $A^\fa$, with coefficients in $H$, by
$$
R_{\lambda, \fa}(P) = c_\fa(\lambda) \prod_{M \in V_\lambda} (x(P) - x(M))^{-1}
$$
where $V_\lambda$ denotes any set of representatives of the non-zero $\lambda$-division points on $A^\fa$ modulo $\wpair{\pm 1}$, and  $c_\fa(\lambda)$ is a unique $12$-th root in $H$ of $\Delta(\CL_\fa)^{N\lambda}/\Delta(\lambda^{-1}\CL_\fa)$ (see also Proposition 1 of the Appendix of \cite{JC2}). Here $\Delta$ denotes Ramanujan's $\Delta$-function. For each non-zero integral ideal $\fb$ of $K$ with $(\fb,\lambda\fq) = 1$, it is easily seen that we have (see Theorem 4 of the Appendix of \cite{JC2})
\begin{equation}\label{2.2}
R_{\lambda, \fa\fb}(\eta_{A^\fa}(\fb)(P)) = \prod_{U \in A^\fa_\fb} R_{\lambda, \fa} (P \oplus U).
\end{equation}

We introduce the index set $\CI$ consisting of all finite sets $\rho = \{(\lambda_i, n_i) \mid i=1, \cdots, r\}$ where $ r \geq 2$, $n_i \in \BZ$, $\lambda_i \neq \pm 1$ non-zero elements of $\CO$ with $(\lambda_i, 6\fq)=1$, and satisfying $\sum_{i=1}^r n_i(N\lambda_i-1) = 0$. Here $N\lambda_i$ denotes the norm from $K$ to $\BQ$ of $\lambda_i$. Given $\rho \in \CI$, we consider the product
\begin{equation}\label{2.3}
\fR_{\rho,\fa}(P) = \prod_{i=1}^r R_{\lambda_i, \fa}(P)^{n_i},
\end{equation}
which is also a rational function on $A^\fa/H$. Under the Weierstrass isomorphism, this rational function can be considered as a function on $\BC/\CL_\fa$ with variable $z$. Taking the derivative logarithm of this function, we have the following result.

\begin{prop}\label{prop2.1}
We have
$$
\frac{d}{dz} \log \fR_{\rho,\fa}(P) = \sum_{i=1}^r \sum_{k=2,\,\text{even}}^\infty -n_i \frac{\phi^k(\fa)}{\xi(\fa)^k\Omega_\infty^k} \left(N\lambda_i - \lambda_i^k \right) L(\bar{\phi}^k, \sigma_\fa, k)z^{k-1}.
$$
In particular, for each even integer $k > 0$, we have
\begin{equation}\label{2.4}
\sfrac{d}{dz}^k \log \fR_{\rho,\fa}(P) \Big\rvert_{z=0} = B_\rho(k)(k-1)! \frac{\phi^k(\fa)}{\xi(\fa)^k\Omega_\infty^k} L(\bar{\phi}^k, \sigma_\fa, k)
\end{equation}
where $B_\rho(k) = \sum_{i=1}^r -n_i \left(N\lambda_i - \lambda_i^k\right)$.
\end{prop}

\begin{proof}
We recall the basic properties of Kronecker--Eisenstein series and elliptic functions, which are fully discussed in \cite{GS}. Let $z$ and $s$ be complex variables. For any lattice $L$ in $\BC$, we define the Kronecker--Eisenstein series by
$$
H_k(z,s,L) = \sum_{w\in L} \frac{(\bar{z}+\bar{w})^k}{|z+w|^{2s}}
$$
where the sum is taken over all $w\in L$, except $-z$ if $z\in L$. It defines a holomorphic function of $s$ in the half plane $\Re(s)>1+k/2$, and has an analytic continuation to the whole $s$-plane. In particular, for each $k \geq 3$, $G_k(L) = H_k(0,k,L)$ is a classic holomorphic Eisenstein series of weight $k$. For the convention, we will denote by
$$
G_1(L) = 0,~G_2(L) = \lim_{s \rightarrow 0+} \sum_{w \in L\setminus \{0\}} w^{-2} |w|^{-2s}.
$$
Let $\sigma(z,L)$ denote the Weierstrass $\sigma$-function. We define a non-holomorphic function $\theta(z,L)$ by
$$
\theta(z,L) = \exp\left(-G_2(L) \frac{z^2}{2}\right) \sigma(z,L).
$$
Then $\theta$ possesses a Taylor expansion of the logarithmic derivative of $\theta(z,L)$ as
$$
\frac{d}{dz} \log \theta(z, L) = \sum_{k=1}^\infty (-1)^{k-1} G_k(L) z^{k-1} = \sum_{k=2,~\text{even}}^\infty -G_k(L) z^{k-1},
$$
where the second equality follows from $G_k(L) = 0$ for $k$ odd.

Moreover, we have the identity
\begin{equation}\label{2.5}
\theta^2(z,L)^{N\lambda}/\theta^2(z,\lambda^{-1}L) = \prod_{0 \neq w \in \lambda^{-1}L/L} (\wp(z,L) - \wp(w,L))^{-1}
\end{equation}
for any non-zero element $\lambda$ of $\CO$. Hence, one gives another expression of the rational function $\fR_{\rho,\fa}(P) = \fR_{\rho,\fa}(\fM(z,\CL_\fa))$ as
$$
\fR_{\rho,\fa}(\fM(z,\CL_\fa))^2 = \prod_{i=1}^r \left(c_\fa(\lambda_i) \frac{\theta^2(z,\CL_\fa)^{N\lambda_i}}{\theta^2(z, \lambda_i^{-1}\CL_\fa)}\right)^{n_i}.
$$
It follows that
\begin{align}\label{2.6}
\begin{split}
\frac{d}{dz}\log\fR_{\rho,\fa}(\fM(z,\CL_\fa)) & = \sum_{i=1}^r n_i \left( N\lambda_i \frac{d}{dz}\log \theta(z, \CL_\fa) - \frac{d}{dz}\log \theta(z, \lambda_i^{-1}\CL_\fa)\right)\\
& = \sum_{i=1}^r \sum_{k=2,\,\text{even}}^\infty -n_i\left(N\lambda_i G_k(\CL_\fa) - \lambda_i^k G_k(\CL_\fa) \right) z^{k-1}.
\end{split}
\end{align}

Finally, Proposition 5.5 in \cite{GS} shows that the partial Hecke $L$-function $L(\bar{\phi}^k, \sigma_\fa, s)$ decomposes into Kronecker--Eisenstein series $H_k(z,s,\CL_\fa)$. In particular, we have
$$
G_k(\CL_\fa) = \frac{\phi^k(\fa)}{\xi(\fa)^k \Omega_\infty^k} L(\bar{\phi}^k, \sigma_\fa, k).
$$
Applying this equality to \eqref{2.6}, this completes the proof of the proposition.
\end{proof}

Now we define the rational function $\fI_{\rho,\fa}(P)$ on $A/H$ by
$$
\fI_{\rho,\fa}(P) = \fR_{\rho,\fa}(P)^2/\fR_{\rho,\fa\fp}(\eta_{A^\fa}(\fp)(P)).
$$
Clearly, it follows from \eqref{2.2} that
$$
\prod_{V \in A_\fp^\fa} \fI_{\rho,\fa}(P \oplus V) = 1.
$$

Let $v$ be the prime of $H$ lying above $\fp$. Let $\fm_v$ be the maximal ideal of the ring $\CO_v$ of integers of the completion $H_v$. For the elliptic curve $A^\fa/H$, we denote by $\wh{A^{\fa, v}}$ the formal group of $A^\fa$ at $v$. We denote by $t = -x/y$ the parameter of this formal group.

\begin{lem}\label{lem2.2}
Let $\fD_{\rho,\fa}(t)$ denote the $t$-expansion of the rational function $\fI_{\rho,\fa}(P)$. Then $\fD_{\rho,\fa}(t)$ lies in $1+\fm_v[[t]]$. In particular, we can define $m_{\rho,\fa}(t) = \frac{1}{2}\log(\fD_{\rho,\fa}(t))$, which lies in $\CO_v[[t]]$.
\end{lem}

\begin{proof}
Let $D_{\lambda_i, \fa}(t) = \sum_{n\geq 0} d_n t^n$ denote the $t$-expansion of the rational function $R_{\lambda_i,\fa}(P)$. We use a classical result (see Lemma 23 of \cite{CW2}) that $D_{\lambda_i,\fa}(t)$ is a unit in $\CO_v[[t]]$. Writing
$$
\wh{\eta_{A^\fa}(\fp)}(t): \wh{A^{\fa,v}} \rightarrow \wh{A^{\fa\fp, v}}
$$
for the formal power series induced by the isogeny $\eta_{A^\fa}(\fp)$, we have $\wh{\eta_{A^\fa}(\fp)}(t) \equiv t^2 \mod \fm_v$. Since $(\lambda_i, \fp)=1$, it follows that
$$
D_{\lambda_i, \fa\fp}(\wh{\eta_{A^\fa}(\fp)}(t)) = \sum_{n\geq 0} d_n^{\sigma_\fp} (\wh{\eta_{A^\fa}(\fp)}(t))^n \equiv \sum_{n\geq 0} d_n^2 t^{2n} \mod \fm_v.
$$
Hence the lemma follows immediately, since
$$
D_{\lambda_i, \fa}(t)^2 = (\sum_{n \geq 0} d_n t^n )^2 \equiv \sum_{n\geq 0} d_n^2 t^{2n} \mod \fm_v.
$$
\end{proof}

Let $\CI_\fp$ denotes the ring of integers of the completion of the maximal unramified extension of $K_\fp$. As $\wh{A^v}$ has height 1 as a formal group, there exists an isomorphism over $\CI_\fp$
$$
\beta_v: \wh{\BG}_m \stackrel{\sim}{\longrightarrow} \wh{A^v},
$$
where $\wh{\BG}_m$ denotes the formal multiplicative group with parameter $w$. For each non-zero integral ideal $\fa$ of $K$ with $(\fa, \fp) = 1$, the isogeny $\eta_A(\fa): A \rightarrow A^\fa$ induces an isomorphism from $\wh{A^v}$ onto $\wh{A^{\fa, v}}$, and hence we have an isomorphism over $\CI_\fp$
$$
\beta_v^\fa : \wh{\BG}_m \stackrel{\sim}{\longrightarrow} \wh{A^{\fa,v}}, ~ \beta_v^\fa = \wh{\eta_A(\fa)}\circ \beta_v.
$$
The isomorphism $\beta_v^\fa$ is given by a power series $t = \beta_v^\fa(w)$ with coefficients in $\CI_\fp$. We write $\Omega_{\fa, v}$ for the coefficient of $w$ in this power series.

\begin{lem}\label{lem2.3}
We have $\Omega_{\fa, v} = \xi(\fa)\Omega_v$.
\end{lem}

\begin{proof}
Viewing $z$ as a parameter of the formal additive group $\wh{\BG}_a$, we have the exponential map $\CE(z, \CL)$ of $\wh{A^v}$ is given by the formal power series
$$
t = \CE(z, \CL) = - \frac{2\wp(z,\CL) - (a_1^2+4a_2)/12}{\wp'(z,\CL)-a_1\left(\wp(z,\CL)- (a_1^2+4a_2)/12\right)-a_3}
$$
Similarly, let $\CE(z,\CL_\fa)$ be defined analogously for the formal group $\wh{A^{\fa,v}}$ by using the Weierstrass isomorphism $\fM(z,\CL_\fa)$. By the uniqueness of the exponential map for a formal group, we have
$$
\beta_v(e^{z/\Omega_v}-1) = \CE(z,\CL),~\beta_v^\fa(e^{z/\Omega_{\fa,v}}-1) = \CE(z,\CL_\fa).
$$
On the other hand, as $\eta_A(\fa)(\fM(z,\CL)) = \fM(\xi(\fa)z,\CL_\fa)$, we have $\wh{\eta_A(\fa)}(\CE(z,\CL)) = \CE(\xi(\fa)z,\CL_\fa))$. The lemma then follows by comparing the first coefficients of the last equality on both sides.
\end{proof}

We now define the formal power series $\fB_{\rho,\fa}(w)$ in $\CI_\fp[[w]]$ by
$$
\fB_{\rho,\fa}(w) = m_{\rho,\fa}(\beta_v^\fa(w)),
$$
and let $\nu_{\rho, \fa}$ be the $\CI_\fp$-valued measure on $\BZ_2$ associated to $\fB_{\rho,\fa}(w)$. Indeed, let $\Lambda_{\CI_\fp}(\fG)$ denotes the ring of $\CI_\fp$-valued measures on a profinite group $\fG$. Then $\nu_{\rho,\fa}$ is determined by Mahler's theorem that there exists the ring isomorphism
\begin{equation}\label{2.7}
\CM : \Lambda_{\CI_\fp}(\BZ_2) \stackrel{\sim}{\longrightarrow} \CI_\fp[[w]], ~ \CM(\nu) = \sum_{n\geq 0} \left(\int_{\BZ_2} \binom xn d\nu\right)w^n = \int_{\BZ_2} (1+w)^x d\nu.
\end{equation}
Now we have the inclusion $i: \Lambda_{\CI_\fp}(\BZ_2^\times) \hookrightarrow \Lambda_{\CI_\fp}(\BZ_2)$ given by extending a measure on $\BZ_2^\times$ to $\BZ_2$ by zero. By \eqref{2.2} we have
$$
\sum_{\zeta \in \{\pm 1\}} \fB_{\rho,\fa}(\zeta(1+w)-1) = 0,
$$
whence the measure $\nu_{\rho,\fa}$ belongs to $\Lambda_{\CI_\fp}(\BZ_2^\times)$. Thus the measure $\nu_{\rho,\fa}$ can be viewed as an element of $\Lambda_{\CI_\fp}(\CG)$ via the isomorphism $\chi_\fp: \CG \stackrel{\sim}{\rightarrow} \BZ_2^\times$. For all $k \geq 0$, we have
$$
\int_\CG \chi_\fp^k d\nu_{\rho,\fa} = \int_{\BZ_2} x^k d\nu_{\rho,\fa} = D^k \fB_{\rho,\fa}(w) \Big\rvert_{w=0} = \sfrac{d}{dz}^k \fB_{\rho,\fa}(e^z-1)\Big\rvert_{z=0},
$$
where $D = (1+w)\frac{d}{dw}$. It is equal to
$$
\Omega_{\fa,v}^k \sfrac{d}{dz}^k \fB_{\rho,\fa}(e^{z/\Omega_{\fa, v}}-1)\Big\rvert_{z=0} =
\frac{1}{2} \Omega_{\fa,v}^k \sfrac{d}{dz}^k \log \fI_{\rho,\fa}(\fM(z,\CL_\fa))\Big\rvert_{z=0}.
$$

\begin{lem}\label{lem2.4}
For each even integer $k > 0$, we have
$$
\Omega_v^{-k} \int_\CG \chi_\fp^k d\nu_{\rho,\fa} = B_\rho(k)(k-1)! \phi^k(\fa)\Omega_\infty^{-k}  \left(L(\bar{\phi}^k, \sigma_\fa, k) - \frac{\phi^k(\fp)}{2}L(\bar{\phi}^k, \sigma_\fa \sigma_\fp, k)\right).
$$
\end{lem}

\begin{proof}
We have
\begin{multline*}
\Omega_{\fa,v}^{-k} \int_\CG \chi_\fp^k d\nu_{\rho,\fa} = \sfrac{d}{dz}^k \log \fR_{\rho, \fa}(\fM(z,\CL_\fa)) \Big\rvert_{z=0} \\ - \frac{1}{2}\sfrac{d}{dz}^k \log \fR_{\rho, \fa\fp}(\eta_{A^\fa}(\fp)(\fM(z,\CL_\fa))) \Big\rvert_{z=0}.
\end{multline*}
Note that $\eta_{A^\fa}(\fp)(\fM(z,\CL_\fa)) = \fM(\xi(\fp)^{\sigma_\fa}z, \CL_{\fa\fp})$ and $\xi(\fa\fp) = \xi(\fa)\xi(\fp)^{\sigma_\fa}$. Then the lemma follows from \propref{prop2.1} and \lemref{lem2.3}.
\end{proof}

We now denote by $\fC$ a set of integral ideals $\fa$ of $K$ prime to $\fp\fq$, whose Artin symbols give precisely the Galois group $G = \Gal(H/K)$. There is the relation
$$
L(\bar{\phi}^k\chi, s) = \sum_{\fa \in \fC} \chi(\sigma_\fa)L(\bar{\phi}^k, \sigma_\fa, s),~ \forall \chi \in G^\ast,
$$
where $G^\ast$ denotes the group of Dirichlet characters of $G$. Hence by \lemref{lem2.4}, for each $\chi \in G^\ast$ we have
$$
\Omega_v^{-k} \sum_{\fa \in \fC} \chi(\sigma_\fa) \phi^{-k}(\fa) \int_\CG \chi_\fp^k d\nu_{\rho,\fa} = B_\rho(k) (k-1)! \left(1- \frac{\phi^k\chi^{-1}(\fp)}{2}\right)\Omega_\infty^{-k}L(\bar{\phi}^k\chi, k).
$$

We can interpret the expression on the left hand side of this formula as follows. Write $\msg$ for the Galois group $\Gal(F_\infty/K)$. Let $B_H$ denote the base extension of $B$ to $H$, and let $\rho_\fp$ be the character of $\msg$ which coincides with the character $\chi_\fp$ on $\CG$ and describes the action of $\msg$ on $(B_H)_{\fp^\infty} = \prod_{\fa \in \fC} A_{\fp^\infty}^\fa$ in the following way. First, we identify $\msg$ with $\CG \times G$. Then for $\sigma_\fc \in G$ with $\fc \in \fC$ and $Q \in A_{\fp^n}^\fa$, we have
$$
\rho_\fp(\sigma_\fc)(Q) = \eta_{A^\fa}(\fc)(Q) \in A_{\fp^n}^{\fa\fc}.
$$
Hence, for $g = h \sigma_\fc \in \msg$ with $h \in \CG$, we have
$$
\rho_\fp(g) = \chi_\fp(h)\phi(\fc).
$$
As is shown in \S3 of \cite{BG}, we can fix a prime $\fP$ of $T$ lying above $\fp$ such that $T_\fP = K_\fp$. Since $\phi(\fa)$ is a unit at $\fP$, we note that the value of $\rho_\fp$ is in $K_\fp^\times$.

Now, let $\delta_\fa$ denote the Artin symbol of $\fa$ in $\msg$ so that $\{\delta_\fa \rvert_H\}_{\fa \in \fC} = \fC$. We define
$$
\nu_\chi^\circ = \sum_{\fa \in \fC}\chi(\sigma_\fa)\delta_\fa^{-1}\nu_{\rho,\fa} \in \CI_\fp[[\msg]] = \CI_\fp[G][[\CG]].
$$
Note that, by Lemma I.3.4 of \cite{dS}, it is independent of the choice of representatives of $G$ in $\msg$. It follows that
$$
\Omega_v^{-k} \sum_{\fa \in \fC} \chi(\sigma_\fa) \phi^{-k}(\fa) \int_\CG \chi_\fp^k d\nu_{\rho,\fa} = \Omega_v^{-k} \int_{\msg} \rho_\fp^k d\nu_\chi^\circ.
$$

\begin{thm}\label{thm2.5}
For each $\chi \in G^\ast$, there exists a unique $\CI_\fp$-valued pseudo-measure $\nu_\chi$ on $\msg=\Gal(F_\infty/K)$ such that for each even integer $k>0$, we have
$$
\Omega_v^{-k} \int_{\msg} \rho_\fp^k d\nu_\chi = (k-1)! \left(1- \frac{\phi^k\chi^{-1}(\fp)}{2}\right)\Omega_\infty^{-k} L(\bar{\phi}^k\chi, k).
$$
\end{thm}

This theorem is immediately followed by the next lemma.

\begin{lem}\label{lem2.6}
There exists an $\CI_\fp$-valued measure $\theta_\rho$ on $\msg$ such that
$$
\int_{\msg} \rho_\fp^k d\theta_\rho = B_\rho(k)
$$
for all $k \geq 1$, and the restriction $\theta_\rho$ to $\Gamma = \Gal(F_\infty/F)$ generates the augmentation ideal of $\Lambda_{\CI_\fp}(\Gamma)$.
\end{lem}

\begin{proof}
This lemma is essentially the same with Lemma II. 7 of \cite{BGS}. We can choose an element $\lambda$ in $\CO$ satisfying $(\lambda, 6\fq) = 1$ and
\begin{equation}\label{2.8}
\lambda \equiv 1 \mod \fp^3,~\bar{\lambda} \equiv 1+ 2^2 \mod \fp^3.
\end{equation}
We set $\rho = \{ (\lambda, 1), (\bar{\lambda}, -1) \} \in \CI$ and write $\tau_\lambda$ and $\tau_{\bar{\lambda}}$ for the Artin symbols of the integral ideals $(\lambda)$ and $(\bar{\lambda})$ of $K$ in $\msg$, respectively. Since $B_\rho(k) = \lambda^k - \bar{\lambda}^k$, the measure
$$
\theta_\rho = \tau_\lambda - \tau_{\bar{\lambda}}
$$
satisfies the first condition of the lemma. For the second condition, we fix a topological generator $\gamma$ of $\Gamma$, and write $\tau_\lambda \rvert_\Gamma = \gamma^a$ and $\tau_{\bar{\lambda}} \rvert_\Gamma = \gamma^b$ with $a,b \in \BZ_2$. The congruences \eqref{2.8} imply that $a \in 2\BZ_2$ and $b \not\in 2\BZ_2$. Hence we have
$$
\tau_\lambda \rvert_\Gamma - \tau_{\bar{\lambda}} \rvert_\Gamma = \gamma^a(1-\gamma^{b-a})
$$
where $\gamma^a$ is a unit in $\Lambda_{\CI_\fp}(\Gamma)$, and $(1-\gamma^{b-a}) = (1-\gamma)u$ with $u$ a unit in $\Lambda_{\CI_\fp}(\Gamma)$.
\end{proof}

\medskip
\section{Vanishing of the $\mu$-invariant for the $\fp$-adic $L$-function}

We have constructed the $\fp$-adic $L$-function $\nu_\chi$ in \thmref{thm2.5} for each $\chi \in G^\ast$. Since we deal with the Iwasawa module $X(H_\infty)$, not $X(F_\infty)$, we define a related pseudo-measure on $\Gal(H_\infty/K)$ by using the following lemma.

\begin{lem}\label{lem3.1}
Let $\delta$ be the generator of $\Delta = \Gal(F_\infty/H_\infty)$. We have $(1+\delta)\Lambda_{\CI_\fp}(\msg) = (1+\delta)\Lambda_{\CI_\fp}(\Gal(H_\infty/K))$.
\end{lem}

\begin{proof}
Since $\Lambda_{\CI_\fp}(\msg) = \CI_\fp[\Delta][[\Gal(H_\infty/K)]]$, it suffices to prove that $(1+\delta)\CI_\fp[\Delta] = (1+\delta)\CI_\fp$. Indeed, if $a+b\delta \in \CI_\fp[\Delta]$, then $(1+\delta)(a+b\delta) = (1+\delta)(a+b) \in (1+\delta) \CI_\fp$.
\end{proof}

Hence there exists an $\CI_\fp$-valued pseudo-measure $m_\chi$ on $\Gal(H_\infty/K)$ such that
\begin{equation}\label{3.1}
(1+\delta)\nu_\chi = (1+\delta)m_\chi.
\end{equation}
We define the $\fp$-adic $L$-function of $\chi$ by
$$
L_\fp(s,\chi) = \int_{\Gal(H_\infty/K)} \kappa^s dm_\chi, ~ s\in \BZ_2,
$$
where $\kappa$ is the natural isomorphism of $\Gal(K_\infty/K)$ onto $1+2^2\BZ_2$ with $\gamma \mapsto u$, and we view such a function on $\Gal(H_\infty/K)$ via the natural surjection from $\Gal(H_\infty/K)$ to $\Gal(K_\infty/K)$. It is well-known that this function is an Iwasawa function, i.e. there exists a formal power series $G_\fp(\chi; w) \in \CI_\fp[[w]]$ such that
\begin{equation}\label{3.2}
G_\fp(\chi; u^s-1)/(u^s-1)^e = L_\fp(s,\chi)
\end{equation}
where $e = 0$ or $1$, according as $\chi \neq 1$ or $\chi =1$. The aim of this section is to prove the vanishing of the $\mu$-invariant of $m_\chi$, or equivalently,

\begin{thm}\label{thm3.2}
For each $\chi \in G^\ast$, the formal power series $G_\fp(\chi; w)$ is prime to $2$, i.e. the $\mu$-invariant of $G_\fp(\chi;w)$ vanishes.
\end{thm}

We remark that this vanishing theorem has recently been proven in \cite{OV}, but in the present paper we will clarify it for our situation. We will use the idea of Sinnott \cite{Si} and Schneps \cite{Sch}. Let $\omega$ be the Teichm\"uller character on $\BZ_2^\times$, and for each $x \in \BZ_2^\times$, let $\pair{x} = x/\omega(x)$. Given a formal power series $F(w) \in \CI_\fp[[w]]$, we associate it to a measure $m_F$ via Mahler's theorem \eqref{2.7}. Then there exists a formal power series $\msl(F)(w) \in \CI_\fp[[w]]$ such that
\begin{equation}\label{3.3}
\int_{\BZ_2^\times} \pair{x}^s dm_F(x) = \msl(F)(u^s-1), ~ s\in \BZ_2.
\end{equation}
The $\mu$-invariant of a formal power series $F(w)$ and that of $m_F$ are both denoted by $\mu(F)$. Recall that $\beta_v: \wh{\BG}_m \stackrel{\sim}{\rightarrow} \wh{A^v}$ is the isomorphism of formal groups. Recall also that $\CO_v$ is the ring of integers of $H_v$.

\begin{lem}[Elliptic analogue of Theorem 1 of \cite{Si}]\label{lem3.3}
Let $F(w) \in \CI_\fp[[w]]$ be a formal power series of the form $F(w) = f(\beta_v(w))$, where $f$ is a rational function on $A$ with coefficients in $\CO_v$. Then we have
$$
\mu(\msl(F)) = \mu(\wt{F} + \wt{F} \circ (-1)),
$$
where
$$
\wt{F}(w) = F(w) - \frac{1}{2}\sum_{\zeta \in \wpair{\pm1}}F(\zeta(1+w)-1),~ (F \circ (-1))(w) = F((1+w)^{-1}-1).
$$
\end{lem}

\begin{proof}
Firstly, we may assume that $\wt{F} = F$ and $F \circ (-1) = F$. Indeed, we put $F' = \wt{F} + \wt{F} \circ (-1)$. If the lemma holds for $F'$ then it holds for $F$, since
$$
\msl(F') = 2\msl(F), ~ \wt{F}' + \wt{F}' \circ (-1) = 2F' = 2(\wt{F} + \wt{F} \circ (-1)).
$$
Moreover, we may also assume that $\mu(F)=0$. Indeed, replacing $f$ by $\pi^{-t}f$, where $\pi$ is a uniformizer of $H_v$, both $\mu$-invariants are decreased by $t$. Hence we have to show that $\mu(\msl(F)) = 1$.

By \eqref{3.3}, we have
$$
\msl(F)(u^s-1) = 2\int_{1+2^2\BZ_2} x^s dm_F(x) = 2 G(u^s-1)
$$
where $G(w)$ is the formal power series associated to $m_F\rvert_{1+2^2\BZ_2}$. Since the characteristic function of $1+2^2\BZ_2$ is given by $1_{1+2^2\BZ_2}(u) = \frac{1}{4} \sum_{i=1}^4 \zeta_4^{(1-u)i}$ with $\zeta_4$ a primitive $4$-th root of unity, we have $G(w) = g(\beta_v(w))$ where $g$ is a rational function on $A$ given by
$$
g(t) = \frac{1}{4} \sum_{i=1}^4 \zeta_4^{-i} f(t+t_i),~t_i = \beta_v(\zeta_4^i-1),
$$
with coefficients in the ring of integers of $H_v(A_4)$. We denote by $\pi'$ a uniformizer of $H_v(A_4)$.

Assume that $g \equiv 0 \mod \pi'$, i.e. $\mu(G) > 0$. Clearly, we have $\mu(G \circ (-1)) >0$. By the first assertion, it is easily seen that $m_F = m_F\rvert_{\BZ_2^\times}$ and that $G \circ (-1)$ is associated to $m_F\rvert_{-1+2^2\BZ_2}$. But then
$$
m_F = m_F\rvert_{\BZ_2^\times} = m_F\rvert_{1+2^2\BZ_2} + m_F\rvert_{-1+2^2\BZ_2}
$$
has positive $\mu$-invariant, which contradicts the second assumption that $\mu(F) = 0$. Hence we have $\mu(G) = 0$ and then $\mu(\msl(F))=1$.
\end{proof}

\begin{lem}\label{lem3.4}
For each $\chi \in G^\ast$, we have
$$
2G_\fp(\chi; w) = \msl\left(\sum_{\fa \in \fC} \chi(\sigma_\fa) (\omega^{-1} \ast D\fB_{\rho,\fa})\right)(u^{-1}(1+w)-1) \cdot u_\chi(w)
$$
where $u_\chi(w)$ is a unit in $\CI_\fp[[w]]$. Here, $\omega$ is the Teichm\"uller character on $\BZ_2^\times$ and $\omega^{-1} \ast F$ denotes the formal power series associated to the measure $\omega^{-1} \cdot m_F$.
\end{lem}

\begin{proof}
By \eqref{3.1} it is easy to check that
$$
\int_{\msg} \kappa(\sigma)^s d\nu_\chi(\sigma) = 2 L_\fp(s, \chi).
$$
On the other hand, we recall that $\fB_{\rho,\fa}(w)$ is the formal power series associated to $\nu_{\rho,\fa}$, i.e. $m_{\fB_{\rho,\fa}} = \nu_{\rho,\fa}$. Hence we have
\begin{align*}
\int_{\msg} \kappa(\sigma)^s d\nu_\chi^\circ(\sigma)
& = \sum_{\fa\in \fC} \chi(\sigma_\fa) \int_\CG \kappa(\sigma)^s dm_{\fB_{\rho,\fa}}(\sigma) \\
& =  \int_{\BZ_2^\times} \pair{x}^s dm_{\left(\sum_{\fa \in \fC} \chi(\sigma_\fa)\fB_{\rho,\fa}\right)}(x) \\
& = \int_{\BZ_2^\times} \pair{x}^{s-1} dm_{\left(\sum_{\fa \in \fC} \chi(\sigma_\fa) (\omega^{-1} \ast D\fB_{\rho,\fa})\right)}(x).
\end{align*}
The proof of the lemma is now complete, since the integral on the measure $\theta_\rho$ can be written as $u_\chi(w)^{-1}$ or $u_\chi(w)^{-1}w$ according as $\chi\neq 1$ or $\chi=1$.
\end{proof}

Recall that the formal power series $D\fB_{\rho,\fa}(w)$ is a rational function whose integral power expansion in $z$ is given by
\begin{equation}\label{3.4}
\frac{1}{2}\Omega_v \frac{d}{dz} \log \fI_{\rho,\fa}(\eta_A(\fa)(\fM(z,\CL)).
\end{equation}
By our construction, it is clear that $\wt{D\fB_{\rho,\fa}} = D\fB_{\rho,\fa}$. Moreover, $D\fB_{\rho,\fa}$ and $D\fB_{\rho,\fa} \circ (-1)$ have the same poles, which implies that $D\fB_{\rho,\fa} = D\fB_{\rho,\fa} \circ (-1)$. We also note that $\mu(F) = \mu(\omega \ast F)$. Hence by \lemref{3.3} and \lemref{3.4}, the proof of \thmref{thm3.2} is now complete by the following lemma.

\begin{lem}\label{lem3.5}
For each $\chi \in G^\ast$, we have $\mu\left(\sum_{\fa \in \fC} \chi(\sigma_\fa)D\fB_{\rho,\fa}\right) = 0$.
\end{lem}

\begin{proof}
Recall that $v$ is our fixed prime of $H$ above $\fp$. Let $\wt{A}$ denote the reduced curve modulo $v$. It suffices to show that the reduction modulo $v$ of the function $\sum_{\fa \in \fC}\chi(\sigma_\fa) D\fB_{\rho,\fa}$ has some poles on $\wt{A}$ with non-zero residue modulo $v$.

By \eqref{3.4}, the function $D\fB_{\rho,\fa}$ can be written as a rational function on $A$
\begin{equation}\label{3.5}
\frac{1}{2}\Omega_v \frac{d}{dz} \log \left( \frac{\prod_{i=1}^r R_{\lambda_i,\fa}(\eta_A(\fa)(P))^{2n_i}}{\prod_{i=1}^r R_{\lambda_i,\fa\fp}(\eta_A(\fa\fp)(P))^{n_i}} \right).
\end{equation}
As \eqref{2.2}, we have the relations
$$
R_{\lambda_i, \fa}(\eta_A(\fa)(P)) = \prod_{W\in A_\fa} R_{\lambda_i}(P\oplus W), ~
R_{\lambda_i,\fa}(\eta_A(\fa\fp)(P)) = \prod_{U\in A_{\fa\fp}} R_{\lambda_i}(P \oplus U).
$$
Therefore, \eqref{3.5} is equal to
\begin{multline*}
\frac{1}{2}\Omega_v \sum_{i=1}^r -2n_i \left(\sum_{W\in A_\fa} \sum_{M \in V_{\lambda_i}} \frac{-2y(P\oplus W)+a_1x(P\oplus W) +a_3}{x(P\oplus W) - x(M)}\right) \\
+ \frac{1}{2}\Omega_v \sum_{i=1}^r n_i \left(\sum_{U \in A_{\fa\fp}} \sum_{M \in V_{\lambda_i}} \frac{-2y(P \oplus U)+a_1x(P\oplus U) +a_3}{x(P \oplus U) - x(M)}\right).
\end{multline*}

We now analyze its possible poles of the reduction of this function on $\wt{A}$. For the second term, we see that they could come from the points $M-U$ for all $M \in V_{\lambda_i}$ and $U \in A_{\fa\fp}$. By the $t$-expansions of $x$ and $y$, we can easily compute that the residue at each $M-U$ is equal to $- n_i \Omega_v$. This is a $\fp$-adic unit because we chose $n_i = \pm 1$ in the proof of \lemref{lem2.6}. However, as $A_\fp$ reduces to zero modulo $v$, the residue at such a pole on $\wt{A}$ is a multiple of $2$, and hence reduces to zero modulo $v$.

For the first term, we note that $x$ is an even function, in particular $x(M) = x(-M)$. Thus this term is equal to
$$
-\frac{1}{2}\Omega_v \sum_{i=1}^r n_i \left(\sum_{W\in A_\fa} \sum_{M \in A_{\lambda_i} \setminus \{0\}} \frac{-2y(P\oplus W)+a_1x(P\oplus W) +a_3}{x(P\oplus W) - x(M)}\right).
$$
Clearly the poles must come from the points $M-W$ for $M \in A_{\lambda_i} \setminus \{0\}$ and $W \in A_\fa$. The residue at each $M-W$ is equal to $n_i\Omega_v$, which is a $\fp$-adic unit. Since reduction modulo $v$ is injective on the set of these $M-W$, each of these $M-W$ gives a pole of the reduced function on $\wt{A}$. Note that as $i=1, \cdots, r$, all of these poles on $\wt{A}$ are distinct because $M$ is a non-zero element of $A_{\lambda_i}$.

Hence the set of poles of the reduction of the function $D\fB_{\rho,\fa}$ on $\wt{A}$ is given by the reduction modulo $v$ of
\begin{equation}\label{3.6}
\wpair{M-W \mid M \in A_{\lambda_i} \setminus \{0\},~ W \in A_\fa},
\end{equation}
and their residues are non-zero modulo $v$. Clearly the same is true for the sum $\sum_{\fa \in \fC} \chi(\sigma_\fa) D\fB_{\rho,\fa}$ because each $\chi(\sigma_\fa)$ is an $h$-th root of unity and thus a $\fp$-adic unit.
\end{proof}

\medskip
\section{Vanishing of the $\mu$-invariant for $X(H_\infty)$}

We will show that the Iwasawa invariants of $X(H_\infty)$ and the $\fp$-adic $L$-function $m$ are equal. As a corollary, \thmref{thm1.1} follows immediately from \thmref{thm3.2}. This equality is a well-known result (for example, see \cite{dS}) for the primes $p \neq 2$, but it can easily be extended to $p=2$ in our case, thanks to our assumptions that $2$ splits in $K$ and $(2,h)=1$.

For the remainder of this section, we denote by $\mu$ and $\lambda$ the $\mu$-invariant and the $\lambda$-invariant of $X(H_\infty)$, respectively. Recall that $\Gamma = \Gal(H_\infty/H)$. For each $n \geq 0$, we define $\Gamma_n = \Gamma^{p^n}$ and $H_n = H_\infty^{\Gamma_n}$. We write $M(H_n)$ for the maximal abelian $2$-extension of $H_n$ which is unramified outside of the primes of $H_n$ above $\fp$. Then it is easily seen that the $\Gamma_n$-coinvariants of $X(H_\infty)$ is given by
$$
X(H_\infty)_{\Gamma_n} = \Gal(M(H_n)/H_\infty).
$$
We have the following asymptotic formula of Iwasawa
\begin{equation}\label{4.1}
\ord_2([M(H_n): H_\infty]) = 2^n \mu + \lambda n + c,~ n\gg0,
\end{equation}
where $c \in \BZ$ is a constant independent of $n$. One can compute this $2$-adic valuation using the methods of Coates and Wiles \cite{CW1}. Let $\fP$ be any prime of $H_n$ lying above $\fp$, and let $U_{n, \fP}$ denote the group of principal units of the completion $H_{n, \fP}$. Write $U_n = \prod_{\fP \mid \fp} U_{n,\fP}$ and $\Phi_n = \prod_{\fP \mid \fp} H_{n,\fP}$. Let $E_n$ be the group of units of $H_n$. As $E_n$ is canonically embedded into $U_n$, let $\ov{E}_n$ be the $\BZ_2$-submodule of $U_n$ generated by $E_n$, and let $D_n$ be the $\BZ_2$-submodule of $U_n$ generated by $E_n$ and $(1+2^2)$. Let $R_\fp(H_n)$ denote the $\fp$-adic regulator for $H_n/K$. Let $\Delta(H_n/K)$ denote the discriminant of $H_n/K$, and choose any generator $\Delta_\fp(H_n/K)$ of the ideal $\Delta(H_n/K)\CO_\fp$.

\begin{thm}\label{thm4.1}
We have
$$
\ord_2([M(H_n):H_\infty]) = \ord_2\left( \frac{h(H_n) R_\fp(H_n)}{\omega(H_n)\sqrt{\Delta_\fp(H_n/K)}} \prod_{\fP \mid \fp} (1-(N\fP)^{-1}) \right) + n+2
$$
where $h(H_n)$ is the class number of $H_n$, $\omega(H_n)$ is the number of roots of unity in $H_n$ and $N\fP$ is the absolute norm of $\fP$.

\end{thm}

\begin{proof}
Let $C_n$ denote the \idele class group of $H_n$. Let $Y_n = \bigcap_{m \geq n} N_{H_m/H_n} C_m$. Let $L(H_n)$ be the maximal unramified extension of $H_n$ in $M(H_n)$. Class field theory gives an isomorphism
$$
(Y_n \cap U_n) / \ov{E}_n \stackrel{\sim}{\longrightarrow} \Gal(M(H_n)/L(H_n)H_\infty).
$$
Noting that $L(H_n) \cap H_\infty = H_n$ because $H_\infty/H_n$ is totally ramified at $\fP$, we obtain an exact sequence
$$
\xymatrixcolsep{1pc}\xymatrix{ 0 \ar[r] & (Y_n\cap U_n)/\ov{E}_n \ar[r] & \Gal(M(H_n)/H_\infty) \ar[r] & \Gal(L(H_n)/H_n) \ar[r] & 0}.
$$
It is easy to check (see Lemma 5 and 6 of \cite{CW1}) that $
Y_n \cap U_n = \Ker (N_{\Phi_n/K_\fp}\rvert_{U_n})$ and $\ov{E}_n = \Ker(N_{\Phi_n/K_\fp} \rvert_{D_n})$, which follows that $[Y_n \cap U_n:\ov{E}_n]=[U_n:D_n]$. Using methods analogous to Lemma 7 and Lemma 8 of \cite{CW1}, one can obtain
\begin{equation}\label{4.2}
[U_n: D_n] = \ord_2 \left( \frac{R_\fp(H_n)}{\omega(H_n)\sqrt{\Delta_\fp(H_n/K)}} \prod_{\fP\mid \fp} (N\fP)^{-1} \right) +n+2.
\end{equation}
The theorem now follows on noting $\prod_{\fP \mid \fp} (N\fP)^{-1}$ and $\prod_{\fP \mid \fp} (1-(N\fP)^{-1})$ have the same order, and that $\Gal(L(H_n)/H_n)$ is the $2$-primary part of the ideal class group of $H_n$.
\end{proof}

We now begin the computation of the Iwasawa invariants of our $\fp$-adic $L$-function $m$. Given $n \geq 0$, let $\epsilon$ be a non-trivial character of $\Gal(H_n/K)$, say $\epsilon = \chi \theta$, where $\chi$ is a character of $G$ and $\theta$ is a character of $\Gal(H_n/H)$. Let $\ff_\epsilon$ denote the conductor of $\epsilon$ with $(f_\epsilon) = \ff_\epsilon \cap \BZ$. As before, we define
$$
L_{\fp, \ff_\epsilon}(s, \epsilon) = \int_{\Gal(H_\infty/K)} \epsilon^{-1}\kappa^s dm
$$
where $m$ is the $\fp$-adic $L$-function defined in the previous section.

For each $n \geq 0$, we denote by $\fC_n$ a set of integral ideals $\fa$ of $K$ prime to $\fp\fq$, whose Artin symbols $\tau_\fa$ give precisely the Galois group $G = \Gal(H_n/K)$. For the convention, we take $\fC_0 = \fC$. For $\fa \in \fC_n$, we denote by $\delta(\fa)$ the Siegel unit as defined in II.2.2 of \cite{dS}. We also denote by $\varphi_{\ff_\epsilon}(\fa)$ the Robert's invariant as defined in II.2.6 of \cite{dS}. Then we put
$$
G(\epsilon) = \frac{\theta(\fp^m)}{2^m} \sum_\tau  \chi(\tau)(\tau(\zeta_m))^{-1}$$
where the sum runs over $\tau\in\Gal(H_nK(\mathfrak{p}^{*\infty})/K)$ with $\tau \rvert_{K(\mathfrak{p}^{*\infty})} = (\fp^m,K(\mathfrak{p}^{*\infty})/K)$, $m$ is an integer such that $\fp^m \parallel \ff_\epsilon$, and $\zeta_m$ is a primitive $p^m$-th root of unity. Define also
$$S(\epsilon) = \begin{cases} \sum_{\fa \in \fC_n} \epsilon(\fa)\log(\varphi_{\ff_\epsilon}(\fa)) & \text{if $\ff_\epsilon \neq 1$} \\ \frac{1}{h}\sum_{\fa \in \fC} \epsilon(\fa)\log(\delta(\fa)) & \text{if $\ff_\epsilon =1$.} \end{cases}
$$
Then, following the methods of \cite[Theorem II.5.2]{dS}, we obtain
$$
L_{\fp, \ff_\epsilon}(0, \epsilon) = \frac{-1}{12f_\epsilon \omega_{\ff_\epsilon}} G(\epsilon^{-1}) S(\epsilon) \left(1- \frac{\epsilon^{-1}(\fp)}{2}\right)
$$
where $\omega_{\ff_\epsilon}$ denotes the number of roots of unity in $K$ congruent to $1$ modulo $\ff_\epsilon$.

On the other hand, the analytic class number formula, together with Kronecker's theorem (see \S 0.2.7, \S I.2.2 and \S IV.3.9 (6) of \cite{T}), gives
$$
\frac{h(H_n)R_\mathfrak{p}(H_n)}{\omega(H_n)} = \frac{h R_\mathfrak{p}(K)}{\omega(K)} \prod_{\epsilon \neq 1} \frac{S(\epsilon)}{12f_\epsilon \omega_{\ff_\epsilon}}
$$
Clearly, $R_\mathfrak{p}(K) = 1$ and $\omega(K) = 2$. Thus, by \thmref{thm4.1}, we have
$$
\ord_2([M(H_n):H_\infty]) = \ord_2\left( \frac{1}{\sqrt{\Delta_\fp(H_n/K)}} \prod_{\epsilon \neq 1} \frac{S(\epsilon)}{12f_\epsilon w_{\ff_\epsilon}} \prod_{\fP \mid \fp} (1-(N\fP)^{-1}) \right) + n+1.
$$
Furthermore, we have $\prod_{\fP \mid \fp} (1-(N\fP)^{-1}) = \frac{1}{2} \prod_{\epsilon \neq 1} \left( 1- \frac{\epsilon^{-1}(\fp)}{2}\right)$, and the conductor-discriminant formula gives that $\prod_{\epsilon \neq 1} G(\epsilon)$ is $\Delta_\fp(H_n/K)^{-1/2}$ up to a $\fp$-adic unit. It follows that
$$
\ord_2([M(H_n):H_\infty]) = \ord_2 \left( \prod_{\epsilon \neq 1} L_{\fp, \ff_\epsilon}(0, \epsilon) \right) + n.
$$

Define as before $G_\fp(\epsilon; w) \in \CI_\fp[[w]]$ to be the formal power series associated to $L_{\fp, \ff_\epsilon}(s, \epsilon)$. In particular, we have
$$
L_{\fp, \ff_\epsilon}(0, \epsilon) = G_\fp(\epsilon; 0) = (\theta^{-1}(u)-1)^{-e} G_\fp(\chi; \theta^{-1}(u)-1)
$$
where $e = 0$ or $1$ according as $\chi\neq 1$ or $\chi=1$ and $u$ is a fixed topological generator of $1+4\BZ_2$. Noting that $\ord_2\left(\prod_{\theta\neq 1} (\theta^{-1}(u)-1)\right) = n$, we obtain
\begin{equation}\label{4.3}
\ord_2([M(H_n):H_\infty]) = \ord_2 \left(\prod_{\epsilon \neq 1} G_\fp(\chi; \theta^{-1}(u)-1) \right).
\end{equation}

For each $\chi \in G^\ast$, we denote by $\mu_\chi$ and $\lambda_\chi$ the $\mu$-invariant and $\lambda$-invariant of $G_\fp(\chi; w)$, respectively. We define
$\mu^{\mathrm{an}} = \sum_{\chi \in G^\ast} \mu_\chi$ and $\lambda^{\mathrm{an}} = \sum_{\chi \in G^\ast} \lambda_\chi$. For sufficiently large $n$, Theorem \ref{thm3.2} tells us that $\ord_2\left(G_\fp(\chi; \theta^{-1}(u)-1)\right) = \ord_2((\theta^{-1}(u)-1)^{\lambda_\chi})$, and hence
$$
\ord_2\left(\prod_{\epsilon \neq 1} G_\fp(\chi;\theta^{-1}(u)-1)\right) = \lambda^{\mathrm{an}}n + c'
$$
where $c' \in \BZ$ is a constant independent of $n$. By \eqref{4.1} and \eqref{4.3}, we conclude that
$$
\mu = \mu^{an} = 0,~ \lambda = \lambda^{an},~ c= c'.
$$
This completes the proof of Theorem \thmref{thm1.1}.

\medskip
\section{Numerical examples for the prime $q < 500$}

Before giving numerical examples for $X(H_\infty)$, we point out the following well-known general lemma.

\begin{lem}\label{lem5.1}
Let $K$ be an imaginary quadratic field, and $p$ any rational prime which splits in $K$ and does not divide the class number of $K$. Let $K_\infty$ be the unique $\BZ_p$-extension of $K$ unramified outside one of the primes $\fp$ of $K$ above $p$. Then $K_\infty$ has no non-trivial abelian $p$-extension unramified outside the primes above $\fp$.
\end{lem}

\begin{proof}
In a similar notation to that used for the case $p=2$, let $X(K_\infty)$ be the Galois group over $K_\infty$ of the maximal abelian $p$-extension of $K_\infty$ unramified outside the primes above $\fp$. Then, as usual in Iwasawa theory, we have $X(K_\infty)_\Gamma = \Gal(R/K_\infty)$ where $R$ denotes the maximal abelian $p$-extension which is unramified outside $\fp$. Thus $R$ must be the maximal pro-$p$ extension of $K$ contained in the union of the ray class fields of $K$ modulo $\fp^n$ for all $n \geq 1$. Thus, as $p$ does not divide $h$, class field theory tells us that $\Gal(R/K)$ must be the maximal pro-$p$ quotient of
$$
\varprojlim_n \left( (\mathcal{O}/\fp^n)^\times/\{\pm 1\} \right)
$$
which is isomorphic to $\BZ_p$. Thus $R = K_\infty$, and the proof of the lemma is complete by Nakayama's lemma.
\end{proof}

Clearly, the assumption of the above lemma is valid for our situation when $p=2$ and $K = \BQ(\sqrt{-q})$ with $q$ any prime congruent to $7$ modulo $8$. The simplest example is given by $K = \BQ(\sqrt{-7})$, which has class number $1$, in which case $X(H_\infty) = X(K_\infty) = 0$. Somewhat surprisingly, the numerical calculations below show that we seem to quite often
have $X(H_\infty) = 0$ for arbitrary primes $q \equiv 7 \mod \, 8$. However, we point out that
when $2$ divides the class number of $H$, it is easily seen that we must necessarily have
$X(H_\infty) \neq 0$, and therefore of infinite order. Andrzej Dabrowski has kindly informed us that
2 does divide the class number of $H$ for the primes $q = 751$, $q=1367$ and $q = 1399$.

We give a list of numerical examples for the primes $q < 500$. By using SAGE calculation, we obtain the class numbers of $K$ and $H$ and the $\fp$-adic regulator $R_\fp = R_\fp(H/K)$ for $H/K$. By \thmref{thm4.1}, we then obtain the index $[M(H):H_\infty]$. Recall that $M(H)$ denotes the maximal abelian $2$-extension of $H$ which is unramified outside the primes of $H$ lying above $\fp$. If we have $[M(H):H_\infty] = 0$, Nakayama's lemma implies immediately that $X(H_\infty) = 0$. We note that the prime $q = 431$ is the first example in which $X(H_\infty) \neq 0$.

\begin{longtable}{c c c c c}
$q$ & $h(K)$ & $h(H)$ & $\ord_2(R_\fp)$ & $\ord_2\left([M(H):H_\infty]\right)$ \\
\hline
7 & 1 & 1 & 0 & 0\\
23 & 3 & 1 & 2 & 0\\
31 & 3 & 1 & 2 & 0\\
47 & 5 & 1 & 4 & 0\\
71 & 7 & 1 & 6 & 0\\
79 & 5 & 1 & 4 & 0\\
103 & 5 & 1 & 4 & 0\\
127 & 5 & 1 & 4 & 0\\
151 & 7 & 1 & 6 & 0\\
167 & 11 & 1 & 10 & 0\\
191 & 13 & 1 & 12 & 0\\
199 & 9 & 1 & 8 & 0\\
223 & 7 & 1 & 6 & 0\\
239 & 15 & 1 & 14 & 0\\
263 & 13 & 1 & 12 & 0\\
271 & 11 & 1 & 10 & 0\\
311 & 19 & 1 & 18 & 0\\
359 & 19 & 1 & 18 & 0\\
367 & 9 & 1 & 8 & 0\\
383 & 17 & 1 & 16 & 0\\
431 & 21 & 1 & 25 & 5\\
439 & 15 & 1 & 14 & 0\\
463 & 7 & 1 & 6 & 0\\
479 & 25 & 1 & 24 & 0\\
487 & 7 & 1 & 6 & 0\\
\hline
\end{longtable}

\noindent For example, when $q=23$, by SAGE calculation we obtain $H=\BQ(\alpha)$ where
$$
\alpha^6 - 3 \alpha^5 + 5 \alpha^4 - 5 \alpha^3 + 5\alpha^2 - 3\alpha + 1 = 0.
$$
The two fundamental units are then given by
$$
\alpha^5 - 2 \alpha^4 + 2\alpha^3 - \alpha^2 + 2\alpha, ~ \alpha^4-2\alpha^3 +3 \alpha^2 - 2\alpha + 2,
$$
and the $\fp$-adic regulator $R_\fp$ is given by
$$
2^2 + 2^4 + 2^6 + 2^7 + 2^8 + 2^9 + 2^{10} + 2^{13} + 2^{17} + 2^{20} + O(2^{23}).
$$

\medskip
\section{Proof of \corref{cor1.2} and \thmref{thm1.3}}

Finally, we give simple proofs that \thmref{thm1.1} implies \corref{cor1.2} and \thmref{thm1.3}. For the corollary, let $J/H$ be any quadratic extension and let $J_\infty = JK_\infty$. Define $\Delta = \Gal(J_\infty/H_\infty)$. If $\Delta$ is trivial, then $J \subseteq H_\infty$, and so there is nothing more to prove. Hence we may assume that $\Delta$ is cyclic of order $2$. The group ring $\BZ_2[\Delta]$ is then a commutative local ring with maximal ideal $\fm$ generated by $2$ and $\delta-1$, where $\delta$ denotes the non-trivial element of $\Delta$. We will use Nakayama's lemma which asserts that, for any $\BZ_2[\Delta]$-module $M$, if there are elements $x_1, \cdots, x_m \in M$ whose images in $M/\fm M$ generate $M/\fm M$ over $\BF_2$, then they generate $M$ itself over $\BZ_2$. We note that, by maximality, $M(J_\infty)$ is clearly Galois over $K_\infty$. We have an exact sequence
$$
\xymatrix{
0 \ar[r] & X(J_\infty) \ar[r] & \Gal(M(J_\infty)/H_\infty) \ar[r] & \Delta \ar[r] & 0.
}
$$
Thus, as usual, $\Delta$ acts on $X(J_\infty)$ by inner automorphisms. In particular, it follows from this action that
\begin{equation}\label{6.1}
X(J_\infty)/(\delta-1)X(J_\infty) = \Gal(R/J_\infty)
\end{equation}
where $R$ denotes the maximal abelian extension of $H_\infty$ contained in $M(J_\infty)$. Our claim is that, under the hypothesis that $X(H_\infty)$ is a finitely generated $\BZ_2$-module, $\Gal(R/J_\infty)$ is a finitely generated $\BZ_2$-module. It follows that $X(J_\infty)/\fm X(J_\infty)$ is a finite dimensional vector space over $\BF_2$, and hence, by Nakayama's lemma, $X(J_\infty)$ is a finitely generated $\BZ_2$-module.

Let $S$ be the set of all primes of $H_\infty$, which do not lie above $\fp$, and which are ramified in $J_\infty$. If $S= \emptyset$, $J_\infty$ is contained in $M(H_\infty)$, in particular $M(J_\infty) = M(H_\infty)$, and hence there is nothing to prove. Otherwise, the set $S$ is finite. This is because, by a basic elementary property of the $\BZ_2$-extension $K_\infty/K$, there are only finitely many primes of $K_\infty$ lying above each prime of $K$, and thus the same is true for the primes of $H_\infty$ lying above a prime of $H$. Hence, as there are only finitely many primes of $H$ which ramifies in $J$, it follows that $S$ is finite. Moreover, the inertia subgroup in $\Gal(R/H_\infty)$ of each prime in $S$ must be of order $2$. Now let $R'$ be the fixed field of the subgroup of $\Gal(R/H_\infty)$ generated by the inertia subgroups of all primes in $S$. Obviously, we have
\begin{equation}\label{6.2}
[R:R'] \leq 2^{\#(S)}
\end{equation}
and $\Gal(R/R')$ is annihilated by $2$. But $R'/H_\infty$ is an abelian $2$-extension which is unramified outside $\fp$, and therefore we have $R' \subset M(H_\infty)$. Hence, by our hypothesis and \eqref{6.2}, $\Gal(R/H_\infty)$ is a finitely generated $\BZ_2$-module, and so is $\Gal(R/J_\infty)$. This completes the proof of \corref{cor1.2}.

For \thmref{thm1.3}, let $F$ and $F_\infty$ be the fields defined as in \eqref{1.4}. Then again the same classical argument (cf. the proof of Lemma 2.1 of \cite{CC}) shows that $E$ has good reduction everywhere over $F$. By \corref{cor1.2}, the Galois group $X(F_\infty)$ is a finitely generated torsion module over the Iwasawa algebra $\Lambda(\Gamma)$ of $\Gamma = \Gal(F_\infty/F)$. Hence, followed by classical arguments (for example, see \cite{JC1}), one can easily obtain
$$
S_{\fp^\infty}(E/F_\infty) = \mathrm{Hom}(X(F_\infty), E_{\fp^\infty}).
$$
Then \thmref{thm1.3} clearly follows immediately from Corollary \ref{cor1.2}.

\subsection*{Acknowledgements}
The authors would like to thank John Coates for very helpful comments on this paper. The second author is supported by the SFB 1085 ``Higher invariants'' (University of Regensburg) funded by the DFG. The third author (Y.L.) would like to thank Ye Tian for the constant encouragement.


\end{document}